\documentclass[a4paper,11pt]{amsart}

\usepackage{style}
\usepackage{geometry}

\title[Symplectic rigidity of O'Grady's tenfolds]{Symplectic rigidity of O'Grady's tenfolds}

\author{Luca Giovenzana}
\address[Luca Giovenzana]{Department of Mathematical Sciences, Loughborough University,
Loughborough,
Leicestershire,
LE11 3TU, United Kingdom}
\curraddr{Department of Pure Mathematics\\ University of Sheffield\\ Hicks Building, Hounsfield Road\\ Sheffield, S3 7RH \\ United Kingdom}
\email{l.giovenzana@sheffield.ac.uk}

\author{Annalisa Grossi}
\address[Annalisa Grossi]{Fakultät für Mathematik, TU Chemnitz, Reichenhainer Str. 39, 09126 Chemnitz, Germany}
\curraddr{Laboratoire de mathématiques d’Orsay, Université Paris Saclay, Rue Michel Magat, Bât. 307, 91405 Orsay, France}
\email{annalisa.grossi@universite.paris.saclay.fr}

\author{Claudio Onorati}
\address[Claudio Onorati]{Dipartimento di Matematica, Universit\`a di Roma Tor Vergata, via della Ricerca Scientifica 1, 00133 Roma, Italy}
\curraddr{Dipartimento di Matematica, Università degli Studi di Bologna, Piazza di Porta San Donato 5,
40126 Bologna, Italy}
\email{claudio.onorati@unibo.it}

\author{Davide Cesare Veniani}
\address[Davide Cesare Veniani]{Institut für Diskrete Strukturen und Symbolisches Rechnen, Universität Stuttgart, Pfaffenwaldring 57, 70569 Stuttgart, Germany}
\email{davide.veniani@mathematik.uni-stuttgart.de}

\date{\today}

\makeatletter
\@namedef{subjclassname@2020}{%
 \textup{2020} Mathematics Subject Classification}
\makeatother


\thanks{Luca Giovenzana was supported by Engineering and Physical Sciences
Research Council (EPSRC) New Investigator Award EP/V005545/1 "Mirror
Symmetry for Fibrations and Degenerations". Annalisa Grossi was partially supported by the DFG through the research grant Le 3093/3-2, and partially supported by the European Research Council (ERC) under the European Union’s Horizon 2020 research and innovation programme (ERC-2020-SyG-854361-HyperK). Claudio Onorati gratefully acknowledges support from the PRIN grant with CUP E84I19000500006 and the ERC grant (2017) CoG771507-StabCondEn. Luca Giovenzana, Annalisa Grossi and Claudio Onorati are members of the INdAM group GNSAGA}

\usepackage{framed}

\begin{document}

\begin{abstract}	
We prove that any symplectic automorphism of finite order of an irreducible holomorphic symplectic manifold of O'Grady's \(10\)-dimensional deformation type is trivial.
\end{abstract}

\maketitle

\section{Introduction}
 
O'Grady~\cite{OGrady:desingularized.moduli.spaces} constructs an irreducible holomorphic symplectic manifold \(\widetilde M\) of dimension~\(10\) and second Betti number~\(24\) as a crepant resolution of a certain moduli space~\(M\) of coherent sheaves on a K3 surface. %of rank \(2\) with trivial first Chern class and second Chern class of degree~\(4\). 
Any complex manifold~\(X\) which is deformation equivalent to \(\widetilde M\) is said to be \emph{of type~\(\OG10\)}.

In this short note, we prove the following theorem.

\begin{theorem} \label{thm:main.aut.symp.OG10}
If \(X\) is a complex manifold of type~\(\OG10\), and \(f \in \Aut(X)\) is a symplectic automorphism of finite order, then \(f\) is the identity.
\end{theorem}

Type \(\OG10\) was the only remaining known deformation type of irreducible holomorphic symplectic manifolds that lacked a systematic treatment of symplectic automorphisms in the literature. \autoref{thm:main.aut.symp.OG10} fills this gap. For references to the works on the other types, see the historical note in \cite[§1.1]{Grossi.Onorati.Veniani:sp.bir.transf.OG6}.

A similar theorem holds for symplectic automorphisms of the other sporadic deformation type \(\OG6\) found by O'Grady in dimension~\(6\) (see \cite[Theorem~1.1]{Grossi.Onorati.Veniani:sp.bir.transf.OG6}). 
On the other hand, \autoref{thm:main.aut.symp.OG10} does not hold for birational transformations.
Indeed, it is known for instance that manifolds of type \(\OG10\) can admit symplectic birational involutions (see \autoref{rmk:birat.transf} for a simple lattice theoretical argument, \cite[§7.3]{Beckmann.Oberdieck:equiv.cat.sympl.surf} for a geometrical example, \cite{Marquand:sympl.birat.} for a complete classification, and \cite{Felisetti.Giovenzana.Grossi} for induced symplectic birational involutions).

By a result by Mongardi and Wandel \cite{Mongardi.Wandel:OG.trivial.action}, any automorphism of a manifold~\(X\) of type~\(\OG10\) acting trivially on the second integral cohomology lattice \(\HH^2(X,\IZ)\) is the identity. Thus, the proof of \autoref{thm:main.aut.symp.OG10} consists in showing that there is no symplectic automorphism~\(f\) of finite order acting non-trivially on \(\HH^2(X,\IZ)\). The proof is divided into two parts, according to the action of \(f\) on the discriminant group of \(\HH^2(X,\IZ)\). In the case of trivial action, we use a trick involving the Leech lattice already employed by Gaberdiel, Hohenegger and Volpato \cite{Gaberdiel.Hohenegger.Volpato:symmetries.K3.sigma.models}, and the classification of invariant sublattices of the Leech lattice by Höhn and Mason~\cite{Hoehn.Mason:290.fixed.Leech}. In the case of non-trivial action, we take advantage of some bounds by Rogers~\cite{Rogers:packing.equal.spheres} and Leech~\cite{Leech:notes.sphere.packings} related to the sphere packing problem.

The paper is divided into two sections. In \autoref{sec:preliminaries}, we introduce the notation and recall some known results.
In \autoref{sec:proof}, we prove \autoref{thm:main.aut.symp.OG10}. 

\subsection*{Acknowledgments} We thank the anonymous referees for their useful suggestions.

\section{Preliminaries} \label{sec:preliminaries}

In this section, we recall basic definitions and some known results which we will use in the sequel.

\subsection{Lattices}
In this paper, we use exactly the same conventions and notation for lattices, i.e., finitely generated free \(\IZ\)-modules \(L \cong \IZ^r\) with a non-degenerate symmetric bilinear pairing \((v,w) \mapsto v \cdot w\), as in our previous paper \cite[§2]{Grossi.Onorati.Veniani:sp.bir.transf.OG6}. 

A lattice \(L\) is called \emph{even} if \(v^2 \in 2\IZ\) for all \(v \in L\), and \emph{unimodular} if \(|{\det(L)}| = 1\).
The \emph{dual} of \(L\) is defined as \(L^\vee \coloneqq \{x \in L\otimes \IQ \mid x \cdot y \in \IZ \text{ for all } y \in L\}\). The divisibility of an element~\(v \in L\) is defined by \((v,L) \coloneqq \gcd\{ v\cdot w \mid w \in L\}\).

The \emph{discriminant group} of a lattice \(L\) is the abelian group \(L^\sharp \coloneqq L^\vee/L\), which has order \(|L^\sharp| = |{\det(L)}|\) and inherits from \(L\) a finite quadratic form when \(L\) is even.
The orthogonal groups of \(L\) and \(L^\sharp\) are denoted by \(\OO(L)\) and \(\OO(L^\sharp)\), respectively. The natural homomorphism \(\OO(L) \rightarrow \OO(L^\sharp)\) is denoted \(g \mapsto g^\sharp\). 

Given a subgroup \(G \subset \OO(L)\), we denote the \emph{invariant sublattice} by \(L^G \coloneqq \{v \in L \mid g(v) = v \text{ for all } g\in G\}\), and the \emph{coinvariant sublattice} by \(L_G \coloneqq (L^G)^\perp\). If the subgroup \(G\) is generated by a single element \(g\), then we write \(L^g \coloneqq L^G\) and \(L_g \coloneqq L_G\). 

A \emph{root} of an even negative definite lattice~\(L\) is an element~\(v \in L\) such that \(v^2 = -2\). 
The \emph{length} of a finite group~\(A\) is the minimal number of generators of \(A\) and it is denoted by \(\ell(A)\); similarly, for a prime number \(p\), the \emph{\(p\)-length} of \(A\) is the minimal number of generators of a Sylow \(p\)-subgroup of \(A\) and it is denoted by \(\ell_p(A)\).

We denote by \(\bU\) and \(\bE_8\) the unique even unimodular lattices of signature \((1,1)\) and~\((0,8)\), respectively. We denote by \(\bA_2\) the unique even lattice of signature \((0,2)\) and determinant~\(3\).

\subsection{Torelli theorem}
We define
\[
    \bL \coloneqq 3\bU \oplus 2\bE_8 \oplus \bA_2.
\]

By a result of Rapagnetta~\cite{Rapagnetta:Beauville.form.known.ihs}, the second integral cohomology group \(\HH^2(X,\IZ)\) of any manifold \(X\) of type~\(\OG10\), together with the Beauville--Bogomolov--Fujiki form, is isomorphic to the lattice~\(\bL\).
We define the following subsets of \(\bL\):
\begin{align*}
    \cW^\pex_{\OG10} & \coloneqq \Set{ v \in \bL}{v^2 = -2} \cup \Set{v \in \bL}{v^2 = -6, (v, \bL) = 3}, \\
    \cW_{\OG10} & \coloneqq \cW^\pex_{\OG10} \cup \Set{ v \in \bL}{v^2 = -4} \cup \Set{v \in \bL}{v^2 = -24, (v, \bL) = 3}.
\end{align*}

Geometrically we have the following interpretation of the two sets above. Let \(X\) be a manifold of type~\(\OG10\) and recall that the positive cone of $X$ is the cone $\mathcal{C}_X$ of classes $x\in\HH^{1,1}(X,\IR)$ such that $(x,x)>0$ and $(x,\omega)>0$ for one (and hence all) K\"ahler class $\omega$. Once we fix a marking \(\eta\colon\HH^2(X,\IZ)\to\bL\), the K\"ahler cone (resp.\ the birational K\"ahler cone) is the chamber in
\[ \mathcal{C}_X\setminus_{D\in\HH^{1,1}(X,\IZ)\cap\eta^{-1}(\cW_{\OG10})}D^\perp \quad\left(\mbox{resp.}\quad\mathcal{C}_X\setminus_{D\in\HH^{1,1}(X,\IZ)\cap\eta^{-1}(\cW^\pex_{\OG10})}D^\perp\right) \]
containing a K\"ahler class.
%Geometrically we have the following interpretation of the two sets above. Let \(X\) be a manifold of type~\(\OG10\) and \(\eta\colon\HH^2(X,\IZ)\to\bL\) a marking; then all walls of the K\"ahler cone (resp.\ birational K\"ahler cone) are of the form $D^\perp$ for some \(D \in\HH^{1,1}(X,\IZ)\cap\eta^{-1}(\cW_{\OG10})\) (resp.\ \(D \in\HH^{1,1}(X,\IZ)\cap\eta^{-1}(\cW^\pex_{\OG10})\)). 
%\dav{Questo non è vero detto così, perché potrebbe essere che \(D^\perp\) non ``tocca'' il cono Kähler, e in quel caso non sarebbe un muro del cono Kähler.} \dav{Suggestion: All walls of the Kähler cone of \(X\) are of the form \(D^\perp\) for some \(D \in\HH^{1,1}(X,\IZ)\cap\eta^{-1}(\cW_{\OG10})\).}
%\cla{Giusto. Ho riscritto direttamente in maniera piu' precisa, come va adesso?} \dav{Per me va anche bene; l'unica cosa è che mi sembra difficile da leggere per via di tutti gli apici/pedici. Io scriverei solo la frase che ho suggerito prima. :)}

We can now state a consequence of the Torelli theorem which will be used to prove \autoref{thm:main.aut.symp.OG10}. We say that an isometry \(g \in \OO(\bL)\) is \emph{induced} by an automorphism of a manifold of type~\(\OG10\) if there exists a manifold $X$ of type~\(\OG10\), a marking \(\eta\colon\HH^2(X,\IZ)\to\bL\), and an automorphism \(f\in\Aut(X)\) such that
\[ g=\eta\circ (f^{-1})^*\circ\eta^{-1}, \]
where \(f^*\in\OO(\HH^2(X,\IZ))\) is the pullback of $f$.

% Wall divisors and prime exceptional divisors of manifolds of type \(\OG10\) have been characterized by Mongardi and Onorati~\cite{Mongardi.Onorati:birational.geometry.OG10}. We define the following sets:
% \begin{align*}
%     \cW^\pex_{\OG10} & \coloneqq \Set{ v \in \bL}{v^2 = -2} \cup \Set{v \in \bL}{v^2 = -6, (v, \bL) = 3}, \\
%     \cW_{\OG10} & \coloneqq \cW^\pex_{\OG10} \cup \Set{ v \in \bL}{v^2 = -4} \cup \Set{v \in \bL}{v^2 = -24, (v, \bL) = 3}.
% \end{align*}
% The monodromy group of a manifold of type \(\OG10\) coincides with the subgroup \(\OO^+(\bL)\) of isometries of positive spinor norm \cite{Onorati:OG10.monodromy}. A subgroup \(G \subset \OO(\bL)\) such that \(\bL_G\) is negative definite satisfies \(G \subset \OO^+(\bL)\) (see e.g. \cite[Lemma~2.3]{Grossi.Onorati.Veniani:sp.bir.transf.OG6}). Therefore, a proof analogous to \cite[Theorems~2.16 and~2.17]{Grossi.Onorati.Veniani:sp.bir.transf.OG6} holds for the following theorems.

\begin{theorem} \label{thm:automorphisms}
A finite subgroup \(G \subset \OO(\bL)\) is induced by a group of symplectic automorphisms of a manifold of type~\(\OG10\) if and only if \(\bL_G\) is negative definite and 
\begin{equation} \label{eq:thm:automorphisms}
    \bL_G \cap \cW_{\OG10} = \emptyset. 
\end{equation}
\end{theorem}

\begin{proof}
    The proof is analogous to \cite[Theorem~2.16]{Grossi.Onorati.Veniani:sp.bir.transf.OG6}. 
    Let us sketch it here for the reader's convenience.

    First of all, let \(X\) be a manifold of type~\(\OG10\), and let us fix a marking, i.e., an isometry \(\HH^2(X,\IZ) \cong \bL\). If \(G\subset\OO(\bL)\) is induced by a group of symplectic automorphisms, then \(\bL_G\) is negative definite by \cite[Lemma~2.12]{Grossi.Onorati.Veniani:sp.bir.transf.OG6}. Moreover, any element of~\(G\) must send the Kähler cone into itself. The Kähler cone is described lattice-theoretically in \cite[Theorem~5.5]{Mongardi.Onorati:birational.geometry.OG10}: it coincides with one of the chambers of the positive cone determined by the orthogonal hyperplanes to the elements in \(\cW_{\OG10}\). Condition~\eqref{eq:thm:automorphisms} follows then from \cite[Lemma~2.14]{Grossi.Onorati.Veniani:sp.bir.transf.OG6}.

    For the reverse implication, it is enough to show the existence of one manifold of type~\(\OG10\) on which \(G\) is induced by symplectic automorphisms. 

    Since \(\bL_G\) is negative definite, there exists by \cite[Lemma~2.13]{Grossi.Onorati.Veniani:sp.bir.transf.OG6} a Hodge structure \(\bL \otimes \IC = \bL^{2,0} \oplus \bL^{1,1} \oplus \bL^{0,2}\) with the following properties: the Hodge structure is preserved by~\(G\), \(\bL^{1,1}\cap\bL=\bL_G\), and \(G\) acts as the identity on \((\bL^{2,0} \oplus \bL^{0,2}) \cap \bL\). 
    By the surjectivity of the period map (see \cite[Theorem~8.1]{Huybrechts:basic.results}), there exists a manifold \(X\) of type~\(\OG10\) with a marking \(\HH^2(X,\IZ) \cong \bL\) such that the Hodge structure on \(\HH^2(X,\IC) \cong \HH^2(X,\IZ) \otimes \IC\) corresponds through the marking to the Hodge structure on \(\bL \otimes \IC\). 
    
    Now, by the Hodge-theoretic Torelli theorem \cite[Theorem~1.3]{Markman:survey}, the group \(G\) is induced by a group of automorphisms of \(X\) if and only if it acts via monodromy operators, it preserves the Hodge structure, and it preserves the Kähler cone.

    By \cite[Theorem~5.4]{Onorati:OG10.monodromy}, the monodromy group of a manifold of type~\(\OG10\) is the whole group \(\OO^+(\bL)\). Since \(\bL_G\) is negative definite, we have by \cite[Lemma~2.3]{Grossi.Onorati.Veniani:sp.bir.transf.OG6} that \(G\subset\OO^+(\bL)\), so that the first condition of the Hodge-theoretic Torelli theorem is satisfied. Moreover, \(G\) preserves the Hodge structure of \(X\) by construction, thus satisfying the second condition of the Hodge-theoretic Torelli theorem.

    Again by \cite[Theorem~5.5]{Mongardi.Onorati:birational.geometry.OG10}, the Kähler cone is cut out inside the positive cone by the hyperplanes which are orthogonal to the vectors in \(\cW_{\OG10}\). Since \(\bL^{1,1}\cap\cW_{\OG10}=\bL_G\cap\cW_{\OG10}=\emptyset\), it follows that the Kähler cone coincides with the positive cone, which is preserved by \(G\) because \(G\subset\OO^+(\bL)\). Therefore, also the third condition of the Hodge-theoretic Torelli theorem is satisfied. Since \(G\) acts as the identity on \((\bL^{2,0} \oplus \bL^{0,2}) \cap \bL\) by construction, the automorphisms inducing \(G\) are symplectic.
\end{proof}

For the sake of completeness, we also state the corresponding theorem for symplectic birational transformations. 
%The proof is identical to the proof of \autoref{thm:automorphisms}, the only difference being to consider the set \(\cW^\pex_{\OG10}\) instead of \(\cW_{\OG10}\).
The proof is analogous to the proof of \autoref{thm:automorphisms}, with the difference that we consider the set \(\cW^\pex_{\OG10}\) instead of \(\cW_{\OG10}\) and the birational Kähler cone instead of the Kähler cone.

\begin{theorem} \label{thm:bir.transformations}
A finite subgroup \(G \subset \OO(\bL)\) is induced by a group of symplectic birational transformations of a manifold of type~\(\OG10\) if and only if \(\bL_G\) is negative definite and 
\[
    \bL_G \cap \cW^\pex_{\OG10} = \emptyset.
\]
\end{theorem}
%Note, in particular, that the coinvariant lattice \(\bL_G\) does not contain roots.

\begin{remark} \label{rmk:birat.transf}
Manifolds of type \(\OG10\) with non-trivial symplectic birational transformations do exist. Consider for instance the involution \(g \in \OO(\bL)\) which consists in exchanging the two copies of \(\bE_8\) and is the identity elsewhere. Then, \(\bL_g\) is isomorphic to~\(\bE_8(2)\). In particular, any \(v \in \bL_g\) satisfies \(4 \mid v^2\), so the group generated by \(g\) satisfies the conditions of \autoref{thm:bir.transformations}.
\end{remark}

\subsection{Leech lattice}

The \emph{Leech lattice}, which we denote by \(\Leech\), is the unique even, unimodular, negative definite lattice of rank~\(24\) without roots. 

Let \(\bLambda_{1,25}\) be the unique even unimodular lattice of signature \((1,25)\). The following proposition was originally proved by Gaberdiel, Hohenegger and Volpato~\cite[§B.2]{Gaberdiel.Hohenegger.Volpato:symmetries.K3.sigma.models}. The same argument was reproduced by Huybrechts~\cite[§2.2, p.~398]{Huybrechts:derived.cat.K3.sp.aut.Conway.group}. We sketch it here for the sake of completeness.

\begin{proposition}[\cite{Gaberdiel.Hohenegger.Volpato:symmetries.K3.sigma.models},\cite{Huybrechts:derived.cat.K3.sp.aut.Conway.group}] \label{prop:Huybrechts.2.2}
Let \(L\) be a lattice, and let \(G \subset \OO(L)\) be a subgroup of isometries which acts trivially on the discriminant group \(L^\sharp\), such that \(L_G\) is negative definite and does not contain roots. If there exists a primitive embedding \(L_G \hookrightarrow \bLambda_{1,25}\), then \(G\) is isomorphic to a subgroup \(G \subset \OO(\Leech)\) such that \(L_G \cong (\Leech)_G\).
\end{proposition}
\begin{proof}
    First of all, since \(G\) acts as the identity on the discriminant group of \(L\), it also acts as the identity on the discriminant group of \(L_G\). In particular, thanks to the primitive embedding \(L_G \hookrightarrow \bLambda_{1,25}\), the action of \(G\) on \(L_G\) can be extended to \(\bLambda_{1,25}\) as the identity on \(L_G^\perp\subset\bLambda_{1,25}\). Note that  \(\bLambda_{1,25}^G\cong L_G^\perp\) is a non-degenerate lattice of signature \((1,25-\rank(L_G))\).

    Let us now fix an isometry \(\bLambda_{1,25}\cong\bLambda_{24}\oplus \bU\), where \(\bU\) is the unimodular hyperbolic lattice of rank \(2\). We also fix an isotropic generator \(w\in \bU\) and we see it as an isotropic element of \(\bLambda_{1,25}\). A \emph{Leech root} of \(\bLambda_{1,25}\) is any root \(\delta\) such that \((\delta,w)=1\). The Weyl group \(W \subset \OO(\bLambda_{1,25})\) is the group generated by reflections associated with Leech roots (see \cite[Chapter~27]{Conway.Sloane:sphere.packings.lattices.groups}). We denote by \(\mathcal{C}_0\) the subset of \(\bLambda_{1,25}\otimes\IR\) of positive classes \(x\) such that \((x,\delta)>0\) for any Leech root \(\delta\).

    Since \(L_G\) contains no roots by hypothesis, it follows that there is no root \(\delta\in\bLambda_{1,25}\) such that \(\bLambda_{1,25}^G\subset\delta^\perp\). In particular, this applies to Leech roots. Therefore, up to changing the primitive embedding \(L_G \hookrightarrow \bLambda_{1,25}\) by an element of \(W\), we can assume that \(\mathcal{C}_0\) is fixed by (the \(\IR\)-linear extension of) \(G\), i.e., \(G\subset\OO(\bLambda_{1,25},\mathcal{C}_0)\). By \cite{Borcherds.PhD}, the group \(\OO(\bLambda_{1,25},\mathcal{C}_0)\) is known to fix the isotropic vector \(w \in \bLambda_{1,25}\), hence \(w \in \bLambda_{1,25}^{G} \cong L_G^{\perp}\). 
    It follows that there exists a primitive embedding \(L_G \hookrightarrow \bLambda_{24}\) given by the composition
    \[
        L_G\hookrightarrow w^\perp\twoheadrightarrow w^\perp/\IZ w\cong\bLambda_{24}.
    \]

    Finally, since again \(G\) acts as the identity on the discriminant group of \(L_G\), the action of \(G\) on \(L_G\) can be extended to an action of \(G\) on \(\bLambda_{24}\) such that \(L_G \cong (\Leech)_G\) as claimed.
\end{proof}

\subsection{Sphere packings} 
Let \(L\) be a positive definite lattice of rank~\(n\). Its \emph{minimal norm}~\(\mu\) and \emph{packing radius}~\(\rho\) are defined as follows:
\begin{equation} \label{eq:rho.mu}
    \mu \coloneqq \min\{ v^2 \mid v \in L \setminus \{0\}\} \quad \text{and} \quad \rho \coloneqq \frac 12 \sqrt \mu.
\end{equation}

The \emph{center density} of \(L\) is defined as
\begin{equation} \label{eq:center.density}
    \delta \coloneqq  \frac{\rho^n}{\sqrt{\det L}}.
\end{equation}

For \(n \leq 24\), there exist upper bounds \(b_n\) on the center density \(\delta\) found by Rogers~\cite{Rogers:packing.equal.spheres}, explicitly computed by Leech~\cite{Leech:notes.sphere.packings}, and reproduced by Conway and Sloane in \cite[Table~1.2, p.~15]{Conway.Sloane:sphere.packings.lattices.groups}.
For the reader's convenience, we copied these bounds in \autoref{tab:Rogers.bound}.

\begin{table}[t]
    \centering
    \caption{Rogers' upper bounds on center density \(\delta\), as computed by Leech.}
    \label{tab:Rogers.bound}
    \begin{tabular}{c|llllllll}
        \toprule
        \(n\) & 1 & 2 & 3 & 4 & 5 & 6 & 7 & 8 \\
        \(b_n\) & 0.5 & 0.28868 & 0.1847 & 0.13127 & 0.09987 & 0.08112 & 0.06981 & 0.06326 \\
        \midrule 
        \(n\) & 9 & 10 & 11 & 12 & 13 & 14 & 15 & 16 \\
        \(b_n\) & 0.06007 & 0.05953 &  0.06136 &  0.06559 & 0.07253 & 0.08278 & 0.09735 & 0.11774 \\
        \midrule
        \(n\) & 17 & 18 & 19 & 20 & 21 & 22 & 23 & 24 \\
        \(b_n\) & 0.14624 & 0.18629 & 0.24308 & 0.32454 & 0.44289 & 0.61722 & 0.87767 & 1.27241 \\
        \bottomrule
    \end{tabular}
\end{table}

\section{Proof of \autoref{thm:main.aut.symp.OG10}} \label{sec:proof}

We fix an irreducible holomorphic symplectic manifold \(X\) of type~\(\OG10\) and a symplectic automorphism \(f \in \Aut(X)\) of finite order. We claim that \(f\) acts trivially on the second integral cohomology group \(\HH^2(X,\IZ)\), and is thus the identity, by a result of Mongardi and Wandel~\cite[Theorem~3.1]{Mongardi.Wandel:OG.trivial.action}.

Choose any marking for \(X\), that is, an isometry \(\eta \colon \HH^2(X,\IZ) \rightarrow \bL\), and let 
\(g \coloneqq \eta \circ (f^{-1})^* \circ \eta^{-1} \in \OO(\bL)\)
be the isometry of \(\bL\) induced by \(f\).
%\cla{\cancel{through the marking} by conjugation?}\dav{no, by conjugation non mi piace}.\ann{diamo un nome al marking per esempio \(\eta:H^{2}(X,\IZ) \to \bL\) e poi scriviamo : and let \(g= \eta \circ f^{*} \circ \eta^{-1} \in \OO(\bL)\).}
% \dav{A parte che credo dovrebbe essere \(\eta \circ (f^{-1})^{*} \circ \eta^{-1}\), non ho capito cosa c'è di male in ``through the marking'', dato che avete capito di cosa sto parlando}
% \cla{Vabbe', a parte gli scherzi, il fatto e' che e' una coniugazione, quindi non e' sbagliato e valeva il discorso che come scritto prima "non mi piaceva troppo", pero' sti cavoli, non ho una grossa opinione in merito.}
% \dav{Per me non è un coniugio, perché coniugati sono due elementi \(a,b\) di un gruppo \(G\) tali che \(gag^{-1} = b\) per qualche altro elemento \(g \in G\), mentre qua non stiamo parlando di elementi di uno stesso gruppo.}
% \cla{In realta' coniugio si usa anche in un grupoide, pero' comunque mi sembra il punto minore su cui combattere e depongo felicemente le armi :)}
Suppose that \(g\) is not the identity. We seek a contradiction on the coinvariant lattice \(\bL_g\). For the rest of this section, we will use the fact that  both \(\bL^g\) and \(\bL_g\) are primitive sublattices of \(\bL\).

The discriminant group \(\bL^\sharp\) has order \(|{\det(\bL)}| = 3\). The only non-trivial automorphism of \(\bL^\sharp\), which we denote by \(-\id\), is the one exchanging the two non-trivial elements of \(\bL^\sharp\). Hence, we have \(\OO(\bL^\sharp) = \{\id, -\id\}\).
In the following, we consider the image \(g^\sharp \in \OO(\bL^\sharp)\) and we treat the two cases \(g^\sharp = \id\) and \(g^\sharp = -\id\) in \autoref{sec:trivial.action} and \autoref{sec:non-trivial.action}, respectively, seeking a contradiction.

\subsection{Trivial action on discriminant group} \label{sec:trivial.action}
Suppose that \(g\) acts trivially on the discriminant group \(\bL^\sharp\), i.e., \(g^\sharp = \id\).

First of all, we claim that \(\ell(\bL_g^\sharp) \leq \ell((\bL^g)^\sharp) + \ell(\bL^\sharp)\). In fact, applying \cite[Proposition~1.15.1]{Nikulin:int.sym.bilinear.forms} to the primitive embedding of \(\bL^g\) in \(\bL\), we see that \(\bL_g^\sharp\) is isometric to an abelian group of the form \(\Gamma^\perp/\Gamma\), where \(\Gamma\subset(\bL^g)^\sharp(-1)\oplus\bL^\sharp\) is an isotropic subgroup. Therefore, we have
\[              \ell(\bL_g^\sharp)=\ell(\Gamma^\perp/\Gamma)\leq\ell(\Gamma^\perp)\leq\ell((\bL^g)^\sharp) + \ell(\bL^\sharp), 
\]
as claimed. 
Now, since $\bL^\sharp=\IZ/3\IZ$, and \(\ell(L^\sharp) \leq \rank(L)\) for any lattice \(L\), we get the following chain of inequalities: 
%Note that \(\bL_g\) is negative definite and does not contain roots by \autoref{thm:automorphisms}. 
%Thanks to \cite[Proposition 1.15.1]{Nikulin:int.sym.bilinear.forms}, the length of the abelian group \(\bL_g^\sharp\) satisfies
\[
    \ell(\bL_g^\sharp) \leq \ell((\bL^g)^\sharp) + \ell(\bL^\sharp) \leq \rank(\bL^g) + 1 = (24 - \rank(\bL_g)) + 1 < \rank(\bLambda_{1,25}) - \rank(\bL_g).
\]
It follows from \cite[Corollary~1.12.3]{Nikulin:int.sym.bilinear.forms} that there exists a primitive embedding \(\bL_g \hookrightarrow \bLambda_{1,25}\).

Note, moreover, that \(\bL_g\) is negative definite and does not contain roots by \autoref{thm:automorphisms}.
Hence, all hypotheses of \autoref{prop:Huybrechts.2.2} are satisfied, 
and we conclude that there exists an element \(g \in \OO(\Leech)\) of order~\(p\) such that \(\bL_g \cong (\Leech)_g\).

Höhn and Mason observe that all coinvariant lattices of the (positive definite) Leech lattice have minimal norm \(4\) (see \cite[p.\ 633]{Hoehn.Mason:290.fixed.Leech}). Thus, \((\Leech)_g\) contains an element of square~\(-4\), and so does \(\bL_g\). Condition~\eqref{eq:thm:automorphisms} of \autoref{thm:automorphisms} is therefore not satisfied, and we have a contradiction.

\subsection{Non-trivial action on discriminant group} \label{sec:non-trivial.action}
Suppose now that \(g\) acts non-trivially on the discriminant group \(\bL^\sharp\), i.e., \(g^\sharp = -\id\). Since \(-\id\) has order \(2\), the order of \(g\) is necessarily even. Up to taking powers, we can suppose that the order of \(g\) is \(2^r\), for some integer \(r\geq 1\). If \(r>1\), then \(g^2\) acts trivially on the discriminant group, and, therefore, it is trivial  by \autoref{sec:trivial.action}. Hence, from now on we can further assume that the order of \(g\) is exactly \(2\). In particular, \(g\) acts as the identity on the invariant sublattice \(\bL^g\), and as multiplication by \(-1\) on the coinvariant sublattice \(\bL_g\).

By work of Nikulin \cite[Proposition~1.5.1]{Nikulin:int.sym.bilinear.forms}, the primitive embedding \(\bL^g \hookrightarrow \bL\) is given by a subgroup \(H\) of \((\bL^g)^\sharp\), a subgroup \(H'\) of \(\bL_g^\sharp\), and an isometry \(\gamma \colon H \rightarrow H'(-1)\), which is called `gluing isometry' in \cite[§2.2]{Grossi.Onorati.Veniani:sp.bir.transf.OG6}. 
By Nikulin's construction, one identifies \(\Xi \coloneqq \bL/(\bL^g\oplus\bL_g)\) with an isotropic subgroup \(\Xi \subset (\bL^g)^\sharp \oplus \bL_g^\sharp\). By definition, the groups \(H\) and \(H'\) are the image of \(\Xi\) under the projections to \(  (\bL^g)^\sharp\) and \(\bL_g^\sharp\), respectively. Therefore, the isometry \(\gamma\) is equivariant with respect to the action of \(g^\sharp\).
%Therefore, the isometry \(\gamma\) is equivariant with respect to the action of \(g\), which is the identity on~\(\bL^g\) and multiplication by \(-1\) on~\(\bL_g\). 
Hence, for every \(\xi \in H\) we have
\[
    \xi = g^\sharp(\xi) = \gamma^{-1}(\gamma(g^\sharp(\xi))) = \gamma^{-1}(g^\sharp(\gamma(\xi))) = \gamma^{-1}(-(\gamma(\xi)) = -\xi,
\]
so all elements of \(H\) have order~\(2\), that is, \(H\) is a \(2\)-elementary abelian group, say of length~\(\ell\). 
In particular, \(|H| = 2^\ell\).
Since \(H\) and \(H'\) are isomorphic, also \(|H'| = 2^\ell\).

% If we denote by \(\Gamma\) the push-out of \(\gamma\) in \((\bL^g)^\sharp \oplus \bL_g(-1)^\sharp\), then we have an identification \(\bL^\sharp \cong \Gamma^\perp/\Gamma\) of finite quadratic forms (cf. \cite[eq. (3)]{Grossi.Onorati.Veniani:sp.bir.transf.OG6}). %If we had \(H' = \bL_g^\sharp\), then we could identify \(\bL^\sharp \cong H^\perp\), where \(H^\perp\) is taken in \((\bL^g)^\sharp\). 
% If we had \(H' = \bL_g^\sharp\), then, writing every element of \(\Gamma^\perp\) as \((x,y)\) with \(x \in (\bL^g)^\sharp\) and \(y \in \bL_g^\sharp\), we could define
%     \[
%         \Gamma^\perp \longrightarrow H^\perp, \quad (x,y) \longmapsto x - \gamma^{-1}(y),
%     \]
%     which would be a surjective homomorphism with kernel \(\Gamma\). Hence, we could identify \(\bL^\sharp \cong \Gamma^\perp/\Gamma \cong H^\perp\).
% In particular, the action of \(g\) on \(\bL^\sharp\) would be trivial, as the action of \(g\) on \((\bL^g)^\sharp\). This contradicts our assumption, therefore we have \(H' \neq \bL_g^\sharp\).

By \cite[eq. (5)]{Grossi.Onorati.Veniani:sp.bir.transf.OG6}, it holds \(|H|^2 \cdot |{\det(\bL)}| = |{\det(\bL^g) \cdot \det(\bL_g)}|\),
which can be written as
\[
    3 = |{\det(\bL)}| = [(\bL^g)^\sharp:H] \cdot [\bL_g^\sharp:H'].
\]
% By the previous paragraph, we have \([\bL_g^\sharp:H'] > 1\), so \([\bL_g^\sharp:H'] = 3\) and \([(\bL^g)^\sharp:H] = 1\). 
%In the former case, though, the action of \(g\) on \(\bL^\sharp\) is trivial, because \(\bL_g^\sharp = H'\) is \(2\)-elementary.
In particular, we have \([\bL_g^\sharp:H'] \leq 3\) and, therefore,
\[
    |{\det(\bL_g)}| = |\bL_g^\sharp| = [\bL_g^\sharp:H'] \cdot |H'| \leq 3 \cdot 2^\ell.
\]
Put \(n \coloneqq \rank(\bL_g)\). Note that \(\ell \leq \ell_2((\bL^g)^\sharp) \leq \rank \bL^g = 24-n\) and \(\ell \leq \ell_2(\bL_g^\sharp) \leq \rank \bL_g =n\).
Thus, we obtain the following upper bound:
\begin{equation} \label{eq:non-triv.action.upper.bound}
    |{\det(\bL_g)}| \leq 3 \cdot 2^{\min(n,24-n)}.
\end{equation}

We now look for a lower bound. By Condition~\eqref{eq:thm:automorphisms} of \autoref{thm:automorphisms}, \(\bL_g\) does not contain elements of square \(-2\) or \(-4\). Since \(\bL_g(-1)\) is an even (positive definite) lattice, its minimal norm \(\mu\) satisfies
\begin{equation} \label{eq:mu>=6}
    \mu \geq 6.
\end{equation}

We let \(b_n\) be Rogers' bound on the center density in dimension \(n\), as in \autoref{tab:Rogers.bound}. Then, denoting \(\delta\) and \(\rho\) respectively the center density and the packing radius of \(\bL_g(-1)\), it follows from \eqref{eq:rho.mu}, \eqref{eq:center.density} and \eqref{eq:mu>=6} that
\begin{equation} \label{eq:non-triv.action.lower.bound}
    |{\det(\bL_g)}| = \frac{\rho^{2n}}{\delta^2} \geq \frac{\rho^{2n}}{b_n^2} = \frac{\mu^n}{2^{2n}b_n^2} \geq \frac{3^n}{2^nb_n^2}.
\end{equation}

Since \(\bL_g\) is negative definite by \autoref{thm:automorphisms}, and \(\bL\) has signature \((3,21)\), we have \(n \leq 21\).
By plugging in the values for \(b_n\) given by \autoref{tab:Rogers.bound}, we see that  
\[
    3 \cdot 2^{\min(n,24-n)} < \frac{3^n}{2^nb_n^2} \quad \text{for all \(n \in \{2,\ldots,21\}\)},
\]
so the two bounds \eqref{eq:non-triv.action.upper.bound} and \eqref{eq:non-triv.action.lower.bound} contradict each other. 
Hence, \(\bL_g\) can only exist in rank \(n = 1\), in which case the two bounds coincide and, necessarily, \(\bL_g \cong [-6]\).
Let \(v\) be a generator of \(\bL_g\) and \(w\) be any other element in \(\bL\). Since \(g\) is an involution, \(w + g(w) \in \bL^g\) and \(w-g(w) \in \bL_g\). Therefore, given that \((v,\bL_g) = 6\), we have
\[
    (v,w) = \frac 12 (v, w+g(w)) + \frac 12 (v,w-g(w)) = \frac 12 (v,w-g(w)) \in 3\IZ,
\]
i.e., \(3 \mid (v,\bL)\). Since \(v\) is a primitive vector of \(\bL\), we have \((v,\bL) \mid \det\bL\), hence \((v,\bL) = 3\). In particular, Condition \eqref{eq:thm:automorphisms} of \autoref{thm:automorphisms} is not satisfied because \(v \in \bL_g \cap \cW_{\OG10}\). Thus,  case \(n = 1\) cannot occur either, and this finishes the proof of \autoref{thm:main.aut.symp.OG10}. \qed

\bibliographystyle{amsplain}
\bibliography{references}

\providecommand{\bysame}{\leavevmode\hbox to3em{\hrulefill}\thinspace}
\providecommand{\MR}{\relax\ifhmode\unskip\space\fi MR }
\providecommand{\MRhref}[2]{%
  \href{http://www.ams.org/mathscinet-getitem?mr=#1}{#2}
}
\providecommand{\href}[2]{#2}
\begin{thebibliography}{10}

\bibitem{Beckmann.Oberdieck:equiv.cat.sympl.surf}
Thorsten Beckmann and Georg Oberdieck, \emph{Equivariant categories of
  symplectic surfaces and fixed loci of {B}ridgeland moduli spaces}, Algebr.
  Geom. \textbf{9} (2022), no.~4, 400--442. \MR{4450620}

\bibitem{Borcherds.PhD}
Richard Borcherds, \emph{The leech lattice and other lattices}, PhD thesis,
  University of Cambridge, 1984.

\bibitem{Conway.Sloane:sphere.packings.lattices.groups}
John~H. Conway and Neil J.~A. Sloane, \emph{Sphere packings, lattices and
  groups}, third ed., Grundlehren Math. Wiss., vol. 290, Springer-Verlag, New
  York, 1999, with additional contributions by E. Bannai, R. E. Borcherds, J.
  Leech, S. P. Norton, A. M. Odlyzko, R. A. Parker, L. Queen and B. B. Venkov.
  \MR{1662447}

\bibitem{Felisetti.Giovenzana.Grossi}
Camilla Felisetti, Franco Giovenzana, and Annalisa Grossi, \emph{O'{G}rady
  tenfolds as moduli spaces of sheaves}, preprint,
  \href{https://arxiv.org/abs/2303.07017}{arXiv:2303.07017} (2023).

\bibitem{Gaberdiel.Hohenegger.Volpato:symmetries.K3.sigma.models}
Matthias~R. Gaberdiel, Stefan Hohenegger, and Roberto Volpato, \emph{Symmetries
  of {K}3 sigma models}, Commun. Number Theory Phys. \textbf{6} (2012), no.~1,
  1--50. \MR{2955931}

\bibitem{Grossi.Onorati.Veniani:sp.bir.transf.OG6}
Annalisa Grossi, Claudio Onorati, and Davide~Cesare Veniani, \emph{Symplectic
  birational transformations of finite order on {O}'{G}rady's sixfolds}, Kyoto
  J. Math. \textbf{63} (2023), no.~3, 615--639.

\bibitem{Hoehn.Mason:290.fixed.Leech}
Gerald H\"{o}hn and Geoffrey Mason, \emph{The 290 fixed-point sublattices of
  the {L}eech lattice}, J. Algebra \textbf{448} (2016), 618--637. \MR{3438323}

\bibitem{Huybrechts:basic.results}
Daniel Huybrechts, \emph{Compact hyper-{K}\"{a}hler manifolds: basic results},
  Invent. Math. \textbf{135} (1999), no.~1, 63--113. \MR{1664696}

\bibitem{Huybrechts:derived.cat.K3.sp.aut.Conway.group}
\bysame, \emph{On derived categories of {K}3 surfaces, symplectic automorphisms
  and the {C}onway group}, Development of moduli theory---{K}yoto 2013, Adv.
  Stud. Pure Math., vol.~69, Math. Soc. Japan, Tokyo, 2016, pp.~387--405.
  \MR{3586514}

\bibitem{Leech:notes.sphere.packings}
John Leech, \emph{Notes on sphere packings}, Canadian J. Math. \textbf{19}
  (1967), 251--267. \MR{209983}

\bibitem{Markman:survey}
Eyal Markman, \emph{A survey of {T}orelli and monodromy results for
  holomorphic-symplectic varieties}, Complex and differential geometry,
  Springer Proc. Math., vol.~8, Springer, Heidelberg, 2011, pp.~257--322.
  \MR{2964480}

\bibitem{Marquand:sympl.birat.}
Lisa Marquand and Stevell Muller, \emph{Classification of symplectic birational
  involutions of manifolds of {OG}10 type}, preprint,
  \href{https://arxiv.org/abs/2206.13814}{arXiv:2206.13814} (2022).

\bibitem{Mongardi.Onorati:birational.geometry.OG10}
Giovanni Mongardi and Claudio Onorati, \emph{Birational geometry of irreducible
  holomorphic symplectic tenfolds of {O}'{G}rady type}, Math. Z. \textbf{300}
  (2022), no.~4, 3497--3526. \MR{4395101}

\bibitem{Mongardi.Wandel:OG.trivial.action}
Giovanni Mongardi and Malte Wandel, \emph{Automorphisms of {O}'{G}rady's
  manifolds acting trivially on cohomology}, Algebr. Geom. \textbf{4} (2017),
  no.~1, 104--119. \MR{3592467}

\bibitem{Nikulin:int.sym.bilinear.forms}
Viacheslav~V. Nikulin, \emph{Integer symmetric bilinear forms and some of their
  geometric applications}, Izv. Akad. Nauk SSSR Ser. Mat. \textbf{43} (1979),
  no.~1, 111--177 (Russian), English translation: Math USSR-Izv. \textbf{14}
  (1979), no. 1, 103--167 (1980). \MR{525944}

\bibitem{OGrady:desingularized.moduli.spaces}
Kieran~G. O'Grady, \emph{Desingularized moduli spaces of sheaves on a {$K3$}},
  J. Reine Angew. Math. \textbf{512} (1999), 49--117. \MR{1703077}

\bibitem{Onorati:OG10.monodromy}
Claudio Onorati, \emph{On the monodromy group of desingularised moduli spaces
  of sheaves on {K3} surfaces}, J. Algebraic Geom. \textbf{31} (2022),
  425--465.

\bibitem{Rapagnetta:Beauville.form.known.ihs}
Antonio Rapagnetta, \emph{On the {B}eauville form of the known irreducible
  symplectic varieties}, Math. Ann. \textbf{340} (2008), no.~1, 77--95.
  \MR{2349768}

\bibitem{Rogers:packing.equal.spheres}
Claude~Ambrose Rogers, \emph{The packing of equal spheres}, Proc. London Math.
  Soc. (3) \textbf{8} (1958), 609--620. \MR{102052}

\end{thebibliography}

\end{document}